\documentclass[10pt,a4paper]{article}
\usepackage[T2A]{fontenc}
\usepackage[utf8]{inputenc}
\usepackage[english]{babel}
\usepackage{amssymb,graphicx}
\usepackage{amsmath,amsfonts}
\usepackage{amsthm}
\usepackage{framed}
\usepackage{fullpage}
\usepackage[makeroom]{cancel}
\usepackage[export]{adjustbox} 
\usepackage{feynmf}
\usepackage{microtype}
\usepackage{hyperref}

 \newtheorem{thm}{Theorem}

  \newtheorem{defn}{Definition}
 \newtheorem{conj}{Conjecture}
 \newtheorem{exam}{Example}

\title{\textbf{On pentagon identity in Ding-Iohara-Miki algebra}}
\author{Yegor Zenkevich\thanks{yegor.zenkevich@gmail.com}\\
  {\small\textit{SISSA, via
      Bonomea 265, 34136 Trieste, Italy,}}\\
  {\small\textit{INFN, Sezione di Trieste,}}\\
  {\small\textit{IGAP, via Beirut 2/1, 34151 Trieste, Italy,}}\\
  {\small\textit{ITEP, Bolshaya Cheremushkinskaya street 25, 117218
      Moscow, Russia,}}\\
  {\small\textit{ITMP MSU, Leninskie gory 1, 119991 Moscow, Russia,}}
}
\date{}
\begin{document}
\maketitle
\vspace{-44ex}
\noindent
{\textit{To T.} \hfill ITEP/TH-34/21}


%
\vspace{37ex}

\begin{abstract}
  We notice that the famous pentagon identity for quantum dilogarithm
  functions and the five-term relation for certain operators related
  to Macdonald polynomials discovered by Garsia and Mellit can both be
  understood as specific cases of a general ``master pentagon
  identity'' for group-like elements in the Ding-Iohara-Miki (or
  quantum toroidal, or elliptic Hall) algebra. We perform some checks
  of this remarkable identity and discuss its implications.
\end{abstract}

\section{Introduction}
\label{sec:introduction}
Pentagon identities are ubiquitous in mathematics and physics. A
non-exhaustive list of their appearances includes representation
theory (where they describe coherence of associators, in particular
quantum $6j$-symbols~\cite{Kirillov:1991ec}), topological field
theory~\cite{Turaev:1992hq, Kashaev:1994pj}, hyperbolic geometry
(Teichm\"uller theory, both classical and
quantum~\cite{QuantTeich}--\cite{Hikami:2001en}) and integrable
systems~\cite{Bazhanov:2007mh}--\cite{Faddeev:1993pe}. In general a
pentagon identity has the form
\begin{equation}
  \label{eq:8}
  F \cdot F = F \cdot F \cdot F,
\end{equation}
where the precise meaning of $F$ depends on the context: for example
it can be an operator with the dot denoting a contraction or
composition, or it can be a simplex of a three-dimensional
triangulation in which case the dot implies gluing along a face.

In the present short note we propose a new pentagon identity for the
generating series of elements of Ding-Iohara-Miki (DIM) algebra
$U_{q_1,q_2,q_3}(\widehat{\widehat{\mathfrak{gl}}}_1)$~\cite{DIM1,
  DIM2}. This algebra has many alternative names because it has been
rediscovered several times in different contexts: elliptic Hall
algebra~\cite{GKV}, spherical double affine Hecke algebra of type
$GL(\infty)$~\cite{SV1}, quantum continuous $\mathfrak{gl}_{\infty}$
algebra~\cite{FFJMM-quant-cont}, deformed
$W_{1+\infty}$-algebra~\cite{DIM2}, quantum toroidal algebra of type
$\mathfrak{gl}_1$~\cite{Hern}. DIM algebra plays an important role in
modern mathematical physics (see
e.g.~\cite{AFS}--\cite{Zenkevich:2020ufs}) and its representation
theory is very rich.

Let us state our conjecture.
\begin{conj}\label{conj:pent}
  Let $\vec{\gamma} \in \mathbb{Z}^2 \backslash \{ (0,0) \}$. Consider
  the formal generating function
  \begin{equation}
    \label{eq:1}
    T_{\vec{\gamma}}(u) = \exp \left(- \sum_{n \geq 1} \frac{(-u)^n}{n} e_{n \vec{\gamma}} \right).
  \end{equation}
  where $e_{\vec{\gamma}} \in U_{q_1, q_2, q_3}
  (\widehat{\widehat{\mathfrak{gl}}}_1)$ with $\vec{\gamma} \in
  \mathbb{Z}^2 \backslash \{ (0,0) \}$ denote the standard generators
  of the DIM algebra (see Definition~\ref{def:dim} for the definition
  of the algebra).  Then the following identity holds
  \begin{equation}
    \label{eq:2}
    \boxed{T_{(0,1)}(v) T_{(1,0)}(u) =  T_{(1,0)}(u) T_{(1,1)}(u v) T_{(0,1)}(v)}
  \end{equation}
\end{conj}
One can view the pentagon identity~\eqref{eq:2} as a $q_3$-deformation
of the simplest of the motivic Kontsevich-Soibelman wall-crossing
formulas~\cite{Kontsevich:2008fj}. Indeed, DIM algebra is a
$q_3$-deformation of the Lie algebra of functions on a quantum torus
(see sec.~\ref{sec:defin-prop-dim}), so for $q_3 = 1$ the
identity~\eqref{eq:2} reduces to a motivic KS wall-crossing formula.

In what follows we give several arguments in favour of
Conjecture~\ref{conj:pent}. If the pentagon identity~\eqref{eq:2} is
valid in the algebra, then it should hold in all of its
representations, of which there are plenty. In
sec.~\ref{sec:rein-homom-or} we prove that in the so-called vector
representation the pentagon identity reduces to the famous
Faddeev-Kashaev-Volkov pentagon identity~\cite{Faddeev:1993pe,
  Faddeev:1993rs} for quantum dilogarithm functions. In
sec.~\ref{sec:pent-ident-gars} we evaluate~\eqref{eq:2} in the Fock
representation of $U_{q_1, q_2, q_3}
(\widehat{\widehat{\mathfrak{gl}}}_1)$ producing an identity for
certain operators related to Macdonald polynomials. This identity has
been proven in the recent work of Garsia and Mellit~\cite{GM}. These
two facts are nontrivial checks of the new pentagon
identity~\eqref{eq:2}. In sec.~\ref{sec:direct-checks} we also perform
some checks in lower orders of expansion in $u$ and $v$ using the
commutation relations of the DIM algebra directly.

Conjecture~\ref{conj:pent} implies some new identities, which we
briefly discuss in sec.~\ref{sec:appl-impl}. For example, there are
homomorphisms of DIM algebra to spherical double affine Hecke algebras
$\mathbb{H}_n$ for any $n$, therefore the pentagon identity should
also be valid in $\mathbb{H}_n$. In more down to earth terms this
would mean an identity involving Ruijsenaars-Schneider Hamiltonians
with $n$ particles. Evaluating~\eqref{eq:2} in the Macmahon
representation one could get an identity for certain combinatorial
operators acting on plane partitions. In principle, one can take any
representation of DIM algebra and obtain the corresponding pentagon
identity, which usually turns out to be new.

We present our conclusions and list some puzzles related to the DIM
pentagon identity in sec.~\ref{sec:conclusions}.

\section{Definition and properties of the DIM algebra}
\label{sec:defin-prop-dim}
\begin{defn}\label{def:dim}
  Let $q_1$, $q_2$ and $q_3$ be formal parameters such that $q_1 q_2
  q_3 = 1$. The algebra
  $U_{q_1,q_2,q_3}(\widehat{\widehat{\mathfrak{gl}}}_1)$ is
  multiplicatively generated by the central elements $c_1$, $c_2$ and
  the elements $e_{\vec{\gamma}}$, with $\vec{\gamma} \in
  \mathbb{Z}^2\backslash \{(0,0)\}$, satisfying the following
  relations (we follow the conventions of~\cite{FJMM1603, BS}):
\begin{enumerate}
\item For $\vec{\gamma} \in \mathbb{Z}^2\backslash \{(0,0)\}$ we
  define\footnote{Notice that $\gcd(\vec{\gamma})$ is always
    positive.} $\gcd(\vec{\gamma}) = \gcd(\gamma_1, \gamma_2)$. We
  call $\vec{\gamma} \in \mathbb{Z}^2\backslash \{(0,0)\}$ coprime if
  $\gcd(\vec{\gamma}) = 1$. Then for any nonzero coprime
  $\vec{\gamma}$
  \begin{equation}
    \label{eq:5}
    [e_{n\vec{\gamma}}, e_{m \vec{\gamma}}] = \frac{n}{\kappa_n}
    \delta_{n+m,0}(c_{n\vec{\gamma}} - c_{-n\vec{\gamma}}),
  \end{equation}
  where we abbreviate $c_{\vec{\gamma}} = c_1^{\gamma_1}
  c_2^{\gamma_2}$ and
  \begin{equation}
    \label{eq:9}
    \kappa_n = (1-q_1^n)(1-q_2^n)(1-q_3^n).
  \end{equation}

\item For $\vec{\alpha}, \vec{\beta} \in \mathbb{Z}^2\backslash
  \{(0,0)\}$, $ \vec{\alpha} \neq \vec{\beta}$ we set $\vec{\alpha} >
  \vec{\beta}$ if $(\alpha_1, \alpha_2)$ and $(\beta_1, \beta_2)$
  are ordered lexicographically, i.e.\ $\alpha_1 > \alpha_2$ or
  $\alpha_1 = \alpha_2$ and then $\beta_1 > \beta_2$. Let
  $\vec{\alpha} \times \vec{\beta} = \alpha_1 \beta_2 - \alpha_2
  \beta_1$ be the two-dimensional cross-product. We call a pair
  $(\vec{\alpha}, \vec{\beta})$ admissible if the triangle
  $T_{\vec{\alpha},\vec{\beta}} = (-\vec{\beta},0,\vec{\alpha})$
  does not contain any $\mathbb{Z}^2$ points in the interior and on at
  least two of its edges. We call the middle $M(T_{\vec{\alpha},
    \vec{\beta}})$ of the triangle $T_{\vec{\alpha},\vec{\beta}}$
  the vertex, which is in the middle according to the lexicographic
  ordering described above. Then for any $\vec{\alpha}, \vec{\beta}
  \in \mathbb{Z}^2\backslash \{(0,0)\}$ such that $T_{\vec{\alpha},
    \vec{\beta}}$ is admissible we have
  \begin{equation}
    \label{eq:10}
    [e_{\vec{\alpha}}, e_{\vec{\beta}}] =
    \begin{cases}
      \frac{1}{\kappa_1} c_{M(T_{\vec{\alpha}, \vec{\beta}})}
      h_{\vec{\alpha}+\vec{\beta}}, &
      \vec{\alpha} \times \vec{\beta} >0,\\
      -\frac{1}{\kappa_1} c_{M(T_{\vec{\beta}, \vec{\alpha}})}
      h_{\vec{\alpha}+\vec{\beta}}, & \vec{\alpha} \times \vec{\beta}
      <0,
    \end{cases}
  \end{equation}
  where $h_{\vec{\gamma}}$ is defined using a generating series as
  follows. Let $\vec{\gamma}$ be coprime, then
  \begin{equation}
  \label{eq:11}
  1 + \sum_{n \geq 1} h_{n \vec{\gamma}} z^{-n} = \exp \left( - \sum_{k
      \geq 1} \frac{\kappa_k}{k} e_{k \vec{\gamma}} z^{-k}\right),
\end{equation}
or more explicitly
\begin{equation}
  \label{eq:12}
  h_{n \vec{\gamma}} = (-1)^n E_n(\kappa_k \vec{e}_{k \vec{\gamma}}),
\end{equation}
where $E_n$ are elementary symmetric functions\footnote{We use the
  upper-case letter to avoid confusion with the generators
  $e_{\vec{\gamma}}$.}, so that
\begin{align}
  \label{eq:13}
  h_{\vec{\gamma}} &= - \kappa_1 e_{\vec{\gamma}},\\
  h_{2\vec{\gamma}} &= \frac{\kappa_1^2}{2} e_{\vec{\gamma}}^2 -
  \frac{\kappa_2}{2} e_{2 \vec{\gamma}},\\
  h_{3\vec{\gamma}} &= - \frac{\kappa_1^3}{6} e_{\vec{\gamma}}^3 +
  \frac{\kappa_1\kappa_2}{2} e_{2 \vec{\gamma}} e_{\vec{\gamma}} -
  \frac{\kappa_3}{3} e_{3\vec{\gamma}},\\
  &\cdots 
\end{align}
\end{enumerate}
\end{defn}
Commutation relations for non-admissible triangles are defined
implicitly, i.e.\ they can be obtained by the successive application
of the relations for admissible triangles. In fact the whole algebra
can be generated from just six elements: $e_{(\pm 1, 0)}$, $e_{(0,\pm
  1)}$ and $c_1$, $c_2$. We list some examples of commutation
relations to make Definition~\ref{def:dim} more accessible and to
demonstrate where the nontrivial dependence on $q_i$ and $c_1$, $c_2$
appears.

\begin{exam}
  If $\vec{\alpha}>(0,0)$, $\vec{\beta}>(0,0)$, then
  $M(T_{\vec{\alpha}, \vec{\beta}}) = (0,0)$, so the factor
  $c_{M(T_{\vec{\alpha}, \vec{\beta}})} = 1$ does not appear in
  Eq.~\eqref{eq:10}.  If in addition
  $\gcd(\vec{\alpha}+\vec{\beta})=1$, then
  \begin{equation}
    \label{eq:14}
        [e_{\vec{\alpha}}, e_{\vec{\beta}}] =
    - e_{\vec{\alpha}+\vec{\beta}}
  \end{equation}
\end{exam}

\begin{exam} We give a sample of the simplest commutation relations:
  \begin{align}
    \label{eq:7}
    [e_{(1,0)}, e_{(0,1)}] &= - e_{(1,1)},\\
    [e_{(0,1)}, e_{(-1,0)}] &= - c_{(0,1)} e_{(-1,1)},\\
    [e_{(-1,0)}, e_{(0,-1)}] &= -  e_{(-1,-1)},\\
    [e_{(2,1)}, e_{(0,1)}] &= \frac{\kappa_1}{2} e_{(1,1)}^2 -
    \frac{\kappa_2}{2 \kappa_1} e_{(2,2)}.
  \end{align}
\end{exam}
  
DIM algebra so defined is obviously invariant under any permutation of
the parameters $q_1$, $q_2$ and $q_3$. It is doubly graded, with
$e_{\vec{\gamma}}$ having grades $(\gamma_1, \gamma_2)$. The algebra
also has a large group of automorphisms which contains the universal
cover $\widetilde{SL}(2,\mathbb{Z})$ of $SL(2,
\mathbb{Z})$~\cite{BS}. This universal cover retains the knowledge of
the ``winding'' of a given $\mathbb{Z}^2$ vector around the origin and
fits into the short exact sequence
\begin{equation}
  \label{eq:15}
  0 \to \mathbb{Z} \to \widetilde{SL}(2,\mathbb{Z}) \to SL(2,
  \mathbb{Z}) \to 0.
\end{equation}
More concretely, the generator\footnote{The standard identification of
  $SL(2,\mathbb{Z})$ generators with matrices is
  $S_{\mathrm{standard}} = \left(
    \begin{smallmatrix}
      0& -1\\
    1& 0
    \end{smallmatrix}
\right)$, $T_{\mathrm{standard}} = \left(
    \begin{smallmatrix}
      1& 1\\
    0& 1
    \end{smallmatrix}
  \right)$, so our generators are $S =
  S_{\mathrm{standard}}^{-1}$, $T =
  S_{\mathrm{standard}} T_{\mathrm{standard}}^{-1} S_{\mathrm{standard}}^{-1}$.} $S = \left(
  \begin{smallmatrix}
    0& 1\\
    -1& 0
  \end{smallmatrix}
\right)$ of $SL(2, \mathbb{Z})$, which rotates $\mathbb{Z}^2$ by
$-\frac{\pi}{2}$ is lifted to a generator of
$\widetilde{SL}(2,\mathbb{Z})$ which fourth power is no longer
trivial, but changes winding number by one. The $T=\left(
    \begin{smallmatrix}
      1& 0\\
    1& 1
    \end{smallmatrix}
  \right)$ generator remains unchanged. The action of automorphisms on
  the DIM generators is
\begin{align}
  \label{eq:16}
  \mathcal{S}(e_{\vec{\alpha}}) &=
  \begin{cases}
    e_{S(\vec{\alpha})}, & \alpha_1 \geq 0, \alpha_2 >0, \text{ or }
    \alpha_1 \leq 0, \alpha_2 <0,\\
    c_{-S(\vec{\alpha})} e_{S(\vec{\alpha})}, & \alpha_1 \geq 0,
    \alpha_2 >0, \text{ or }
    \alpha_1 \leq 0, \alpha_2 <0,\\
  \end{cases}\\
\mathcal{S}(c_{\vec{\alpha}}) &=c_{S(\vec{\alpha})} \\
   \mathcal{T}(e_{\vec{\alpha}}) &=
      e_{T(\vec{\alpha})},\\
   \mathcal{T}(c_{\vec{\alpha}}) &=
      c_{T(\vec{\alpha})}.     
\end{align}
Notice that in particular that
\begin{equation}
  \label{eq:19}
  \mathcal{S}^4(e_{\vec{\alpha}}) = c_{-2\vec{\alpha}}
  e_{\vec{\alpha}}
\end{equation}
for any $\vec{\alpha}$, which is precisely the element of winding
number one in $\widetilde{SL}(2,\mathbb{Z})$. The complicated picture
with universal cover arises because of the lexicographic ordering
entering the definition of $c_{M(T_{\vec{\alpha}, \vec{\beta}})}$. In
the absence of central charges the automorphism group is just
$SL(2,\mathbb{Z})$. In fact since the pentagon identity involves only
the generators from the first quadrant, the central charges will never
appear in the commutation process.

In the limit when one of the parameters $q_i$ goes to one, the
algebra $U_{q_1,q_2,q_3}(\widehat{\widehat{\mathfrak{gl}}}_1)$ becomes
the Lie algebra of functions on a quantum torus, also known as
$qW_{1+\infty}$-algebra. Since $q_i$ enter in the definition
symmetrically, we can consider $q_3 \to 1$ without loss of
generality. We introduce the rescaled generators and central charges
in the limit as follows:
\begin{equation}
  \label{eq:20}
  w_{\vec{\alpha}} = \left( q_1^{\frac{\gcd(\vec{\alpha})}{2}} -
    q_1^{-\frac{\gcd(\vec{\alpha})}{2}} \right) e_{\vec{\alpha}},
  \qquad \qquad c_i \to q_3^{\tilde{c}_i} \quad i=1,2.
\end{equation}
Then the commutation relations turn into Lie algebraic form:
\begin{equation}
  \label{eq:21}
 \boxed{ [w_{\vec{\alpha}}, w_{\vec{\beta}}] = \left( q_1^{\frac{\vec{\alpha}
        \times \vec{\beta}}{2}} - q_1^{-\frac{\vec{\alpha} \times
        \vec{\beta}}{2}} \right) w_{\vec{\alpha}+\vec{\beta}} +
  \delta_{\vec{\alpha}+\vec{\beta},0} ( \alpha_1 \tilde{c}_1 + \alpha_2 \tilde{c}_2)}
\end{equation}
Notice that the identities~\eqref{eq:21} with $\tilde{c}_i = 0$
would follow from the multiplication rule
\begin{equation}
  \label{eq:22}
  w_{\vec{\alpha}} w_{\vec{\beta}} = q_1^{\frac{\vec{\alpha} \times \vec{\beta}}{2}} w_{\vec{\alpha}+\vec{\beta}},
\end{equation}
but it does not hold in a general representation of the Lie
algebra~\eqref{eq:21}.

\section{Checks of the pentagon conjecture}
\label{sec:checks-conjecture}
In this section we check Conjecture~\ref{conj:pent} by evaluating it
in a vector representation of DIM algebra
(sec.~\ref{sec:rein-homom-or}), in a Fock representation
(sec.~\ref{sec:pent-ident-gars}) and also by some brute force
calculations using commutation relations
(sec.~\ref{sec:direct-checks}). 

\subsection{Reineke's character, vector representations and quantum
  dilogarithm}
\label{sec:rein-homom-or}

\begin{thm}
  Let $\mathcal{O}_{q_i}$ be the associative algebra of functions on a
  two-dimensional quantum torus $(\mathbb{C}^{\times} \times
  \mathbb{C}^{\times})_{q_i}$ with non-commutativity parameter
  $q_i$. The algebra $\mathcal{O}_{q_i}$ is multiplicatively generated
  by $\mathbf{x}$ and $\mathbf{y}$, satisfying the $q_i$-commutation
  relations
  \begin{equation}
  \label{eq:4}
  \mathbf{y}\, \mathbf{x} = q_i \mathbf{x}\, \mathbf{y}.
\end{equation}
Then the map $\rho_i$ (for $i=1,2,3$) called Reineke's
character~\cite{Reineke}
\begin{equation}
  \label{eq:6}
  \rho_i:
  U_{q_1,q_2,q_3}(\widehat{\widehat{\mathfrak{gl}}}_1) \to
  \mathcal{O}_{q_i},
\end{equation}
described by the formulas
\begin{align}
  \label{eq:3}
  \rho_i(e_{\vec{\gamma}}) &= \frac{q_i^{\frac{\gamma_1
        \gamma_2}{2}}}{q_i^{\frac{\gcd(\vec{\gamma})}{2}}-q_i^{-\frac{\gcd(\vec{\gamma})}{2}}}
  \mathbf{x}^{\gamma_1} \mathbf{y}^{\gamma_2}, \\
  \rho_i(c_{\vec{\gamma}})&=1
\end{align}
for any $ \vec{\gamma} \neq 0$ is a homomorphism.
\end{thm}

\begin{proof}
  By explicit verification of the commutation relations.
\end{proof}

Under the character map $\rho_i$ the elements $T_{\vec{\gamma}}(u)$
for coprime $\vec{\gamma}$ become
\begin{equation}
  \label{eq:17}
  \rho_i(T_{\vec{\gamma}}(u)) = \left( - q_i^{\frac{1+ \gamma_1 \gamma_2}{2}} u\,
    \mathbf{x}^{\gamma_1} \mathbf{y}^{\gamma_2}; q_i \right)_{\infty}^{-1},
\end{equation}
where
\begin{equation}
  \label{eq:18}
  (x;q)_{\infty} = \prod_{n \geq 0} (1- q^i x).
\end{equation}
The pentagon identity~\eqref{eq:2} then becomes the well-known
identity for quantum dilogarithms (we get rid of the
$q_i^{\frac{1}{2}}$ factors by rescaling the generating function
parameters $u \mapsto q_i^{\frac{1}{2}}u$, $v \mapsto
q_i^{\frac{1}{2}} v$):
\begin{equation}
  \label{eq:23}
\boxed{  (-v\, \mathbf{y} ; q_i )_{\infty}^{-1}
  (- u\, \mathbf{x} ; q_i )_{\infty}^{-1} =
  (- u\, \mathbf{x} ; q_i )_{\infty}^{-1}   (- u
  v\, \mathbf{x}\, \mathbf{y} ; q_i )_{\infty}^{-1}
  (- v\, \mathbf{y} ; q_i )_{\infty}^{-1}}
\end{equation}
Thus, under Reineke's map DIM pentagon identity~\eqref{eq:2} maps to a
true statement. It is curious that we can obtain three ``different''
quantum dilogarithm identities, one for each $q_i$, from a single
``master'' identity~\eqref{eq:2}.

The Reineke's character can be used to obtain a representation of the
DIM algebra from a representation of $\mathcal{O}_{q_i}$ e.g.\ using
difference operators on functions of a single variable $w$. One can
take
\begin{align}
  \label{eq:24}
  \mathbf{x} &= q_i^{- w\partial_w},\\
  \mathbf{y} &= w.
\end{align}
These are known as vector representations $\mathcal{V}_{q_i}$ of
$U_{q_1,q_2,q_3} (\widehat{\widehat{\mathfrak{gl}}}_1)$.

Another way to view DIM pentagon identity~\eqref{eq:2} is as a
\emph{refinement,} or \emph{categorification} of the pentagon identity
in $\mathcal{O}_{q_i}$, because it splits the terms with the same
powers of $\mathbf{x}$ and $\mathbf{y}$ into separate entities, e.g.\
a term $\mathbf{x}^2 \mathbf{y}^2$ in the identity for quantum
dilogarithms can be either $e_{(2,2)}$ or $e_{(1,1)}^2$ in the DIM
algebra (see sec.~\ref{sec:direct-checks} for precisely this
computation). The Reineke's character, as its name suggests, plays the
role of taking the trace, or the grading, of a refined formula.

\subsection{Pentagon identities of Garsia-Mellit from Fock
  representation}
\label{sec:pent-ident-gars}
\begin{thm}
  Let $\mathcal{F}_{q_1,q_2}^{(1,0)} = \mathbb{C}[p_1, p_2,
  \ldots]$ be the vector space of polynomials in time variables
  $p_n$. Denote $q = q_1$, $t= q_2^{-1}$ and let
  $M_{\lambda}^{(q,t)}(p_n)$ be Macdonald polynomial in the standard
  normalization labelled by a Young diagram $\lambda = \{ \lambda_1,
  \lambda_2, \ldots, \lambda_{l(\lambda)}\}$. Then the map
  \begin{equation}
    \label{eq:26}
    f:
    U_{q_1,q_2,q_3} (\widehat{\widehat{\mathfrak{gl}}}_1) \to
    \mathrm{Aut} (\mathcal{F}_{q,t^{-1}}^{(1,0)})
  \end{equation}
  defined by relations below gives a representation of
  $U_{q_1,q_2,q_3} (\widehat{\widehat{\mathfrak{gl}}}_1)$ on
  $\mathcal{F}_{q,t^{-1}}^{(1,0)}$:
  \begin{align}
    \label{eq:25}
    f(e_{(0, n)}) &= \frac{1}{1-t^{-n}}n \frac{\partial}{\partial
      p_n},
    \quad n>0,\\
    f(e_{(0,-n)}) &= - \left( \frac{q}{t} \right)^{\frac{n}{2}}
    \frac{1}{1-q^n} p_n, \quad n>0,\\
    f(e_{(1,n)}) &= - \frac{1}{(1-q)(1-t^{-1})} \oint_{\mathcal{C}_0}
    \frac{dz}{z} z^n e^{\sum_{k \geq 1} \frac{z^k}{k} (1-t^{-k}) p_k }
    e^{- \sum_{k \geq 1} z^{-k} (1-q^k) \frac{\partial}{\partial
        p_k}}, \quad n \in
    \mathbb{Z},\\
    f(e_{(-1,n)}) &= \frac{1}{(1-q^{-1})(1-t)} \oint_{\mathcal{C}_0}
    \frac{dz}{z} z^n e^{ - \sum_{k \geq 1} \frac{z^k}{k} (1-t^{-k})
      \left( \frac{t}{q} \right)^{\frac{k}{2}} p_k} e^{ \sum_{k \geq
        1} z^{-k} (1-q^k) \left( \frac{t}{q} \right)^{\frac{k}{2}}
      \frac{\partial}{\partial p_k} }, \quad n \in \mathbb{Z},\\
    f(e_{\pm n,0})M_{\lambda}^{(q,t)}(p_m) & = \pm \left( -
      \frac{1}{(1-q^{\pm n})(1-t^{\mp n})} + \sum_{(i,j)\in \lambda}
      q^{\pm n(j-1)} t^{\pm n(1-i)}\right) M_{\lambda}^{(q,t)}(p_m),
    \qquad n>0,\\
    f(e_{n,1})M_{\lambda}^{(q,t)}(p_m) &= \frac{\left( \frac{t}{q}
      \right)^{\frac{1}{2}(1 + \delta_{n \geq 0})}}{1-q^{-1}}
    \oint_{\mathcal{C}_0} \frac{dz}{z} z^n \sum_{i=1}^{l(\lambda)}
    \frac{1- q \frac{z}{t^{1-i}}}{1 - t \frac{z}{t^{1-i}}}
    \prod_{j=i+1}^{l(\lambda)} \frac{\psi \left(
        \frac{z}{q^{\lambda_j} t^{1-j}} \right)}{\psi \left(
        \frac{z}{t^{1-j}} \right)} \delta \left( \frac{z}{q^{\lambda_i
          -1} t^{1-i}} \right)M_{\lambda - 1_i}^{(q,t)}(p_m),\label{eq:29}\\
    f(e_{n,-1})M_{\lambda}^{(q,t)}(p_m) &= -\frac{\left( \frac{q}{t}
      \right)^{\frac{1}{2}\delta_{n \leq 0}}}{1-q}
    \oint_{\mathcal{C}_0} \frac{dz}{z} z^n \sum_{i=1}^{l(\lambda)+1}
    \prod_{j=1}^{i-1} \psi \left( \frac{z}{q^{\lambda_j} t^{1-j}}
    \right)\delta \left( \frac{z}{q^{\lambda_i} t^{1-i}} \right)
    M_{\lambda + 1_i}^{(q,t)}(p_m), \qquad n \in
    \mathbb{Z},\label{eq:28}
  \end{align}
  \begin{align}
      f(c_1) & = 1,\\
      f(c_2) &= \left( \frac{t}{q} \right)^{\frac{1}{2}}.
  \end{align}
  where $\mathcal{C}_0$ is a small contour around the origin, $\lambda
  \pm 1_i = \{ \lambda_1, \ldots, \lambda_{i-1}, \lambda_i \pm 1,
  \lambda_{i+1}, \cdots \} $ and
  \begin{equation}
    \label{eq:27}
    \delta_{n \geq 0} =
  \begin{cases}
    1, & n\geq 0\\
    0, & n<0
  \end{cases}, \qquad \qquad \delta(x) = \sum_{n \in \mathbb{Z}} x^n,
  \qquad \qquad \psi (x) = \frac{(1-tx) \left( 1 - \frac{q}{t}x \right)}{(1-x)(1-qx)}.
  \end{equation}
\end{thm}
\begin{proof}
  One verifies directly the commutation relations between the six
  generating elements $e_{(\pm 1, \pm 1)}$, $c_1$ and $c_2$.
\end{proof}
Several remarks are in order:
\begin{enumerate}
\item Some expressions in Eqs.~\eqref{eq:25}--\eqref{eq:28} define the
  same operators in two different ways, e.g.\ $e_{(1,0)}$. In these
  cases both expressions are valid due to the properties of Macdonald
  polynomials.

\item The the strange factors like $\left( \frac{t}{q}
  \right)^{\frac{1}{2} \delta_{n \geq 0}}$ in Eqs.~\eqref{eq:29},
  \eqref{eq:28} come from the application of the $\mathcal{S}$
  automorphism to the ``vertical'' Fock representation taken
  from\footnote{We also rescale the generators of~\cite{AFS} to
    conform with our commutation relations, which make the symmetry
    between $q_1$, $q_2$, $q_3$ explicit.  }~\cite{AFS}.
\end{enumerate}

We can directly evaluate the generating series $T_{(1,0)}$ and
$T_{(0,1)}$ in the Fock representation:
\begin{align}
  \label{eq:30}
  f(T_{(1,0)}(u)) M_{\lambda}^{(q,t)}(p_n) &= \exp \left( \sum_{n \geq 1}
    \frac{(-u)^n}{(1-q^n)(1-t^{-n})} \right) \prod_{(i,j)\in \lambda}
  (1 + u q^{j-1} t^{1-i}) M_{\lambda}^{(q,t)}(p_n),\\
  f(T_{(1,0)}(u)) &= \exp \left( -\sum_{n\geq 1} \frac{(-v)^n}{1-t^{-n}}
    \frac{\partial}{\partial p_n} \right).
\end{align}
It is left to find $T_{(1,1)}(uv)$. To do this we notice that the
element $\mathcal{S}^{-1} \mathcal{T}\mathcal{S} \in
\widetilde{SL}(2,\mathbb{Z})$ of the automorphism group of the algebra
preserves the generators $e_{(n,0)}$, which are diagonal in the basis
of Macdonald operators. It also preserves the central charges of the
Fock representation. One can then guess that $\mathcal{S}^{-1}
\mathcal{T}\mathcal{S}$ can be realized as an operator $\tau$ on the
Fock space, diagonal in the basis of Macdonald polynomials:
\begin{equation}
  \label{eq:32}
  f(\mathcal{S}
  \mathcal{T}^{-1}\mathcal{S}^{-1} (e_{\vec{\gamma}})) =
  \tau^{-1}f(e_{\vec{\gamma}}) \tau. 
\end{equation}
In fact it is enough to verify its action on the generators $e_{(0,\pm
  1)}$, $e_{(\pm 1 ,0)}$ to find
\begin{equation}
  \label{eq:33}
  \tau M_{\lambda}^{(q,t)}(p_n) = \left( \prod_{(i,j)\in \lambda} q^{j-1}
  t^{1-i}\right) M_{\lambda}^{(q,t)}(p_n).
\end{equation}
Thus we have
\begin{equation}
  \label{eq:36}
  f(T_{(1,1)}(uv)) = f(\mathcal{S}
  \mathcal{T}^{-1}\mathcal{S}^{-1}(T_{(1,0)}(uv))) = \tau^{-1}
  f(T_{(1,0)}(uv)) \tau.
\end{equation}

To make contact with the work of Garsia and Mellit~\cite{GM} we notice
that their parameters $q_{\mathrm{GM}}$, $t_{\mathrm{GM}}$ and the power
sums $p_n^{\mathrm{GM}}$ are related to ours as follows:
\begin{align}
  \label{eq:31}
  q_{\mathrm{GM}} &= q,\\
  t_{\mathrm{GM}} &= t^{-1},\\
  p_n^{\mathrm{GM}}&= (1-t^n) p_n.
\end{align}
They also define operators $\tau_v^{\mathrm{GM}}$,
$\Delta'^{\mathrm{GM}}_u$ and $\nabla_{\mathrm{GM}}$, which are
related to our $T_{(1,0)}$, $T_{(0,1)}$ and $\tau$ respectively:
\begin{align}
  \label{eq:34}
  \tau_v^{\mathrm{GM}} &= f(T_{(0,1)}(-v/t)),\\
  \Delta'^{\mathrm{GM}}_u &= f(T_{(1,0)}(-u)),\\
  \nabla_{\mathrm{GM}} & =  (-1)^d \tau,
\end{align}
where $d$ counts the total degree of polynomial in
$\mathcal{F}_{q_1,q_2}^{(1,0)}$. We also notice that
\begin{equation}
  \label{eq:35}
  (-1)^d f(T_{(1,1)}(u v)) (-1)^d = f(T_{(1,1)}(-u v)).
\end{equation}
Combining all these changes of variables and shifting $v \to tv$ in
$T_{(0,1)}$ we find that the DIM pentagon identity under the action of
$f$ reduces to
\begin{equation}
  \label{eq:37}
  (\Delta_u'^{\mathrm{GM}})^{-1} \tau_v^{\mathrm{GM}}
  \Delta_u'^{\mathrm{GM}} (\tau_v^{\mathrm{GM}})^{-1} = \nabla^{-1}
  \tau_{uv}^{\mathrm{GM}} \nabla,
\end{equation}
which is precisely what is proven in~\cite{GM}.

This is another nontrivial check of the DIM pentagon identity: in the
Fock space case, unlike in sec.~\ref{sec:rein-homom-or}, there are no
additional multiplicative relations $e_{\vec{\alpha}} e_{\vec{\beta}}
\sim e_{\vec{\alpha}+ \vec{\beta}}$ for $f(e_{\gamma})$.

\subsection{Direct checks}
\label{sec:direct-checks}
One can try to verify some lower orders in the expansion of the
pentagon identity~\eqref{eq:2} in $u$ and $v$. At first one gets some
identities which can be derived from the commutation relations without
deformation parameters $q_i$ making any appearance:
\begin{align}
  \label{eq:38}
  u^1v^1&: \qquad e_{(0,1)} e_{(1,0)} = e_{(1,0)} e_{(0,1)} +
  e_{(1,1)},\\
  u^2v^1&: \qquad \frac{1}{2} e_{(0,1)} \left( e_{(1,0)}^2 - e_{(2,0)}
  \right) = \frac{1}{2} (e_{(1,0)}^2 - e_{(2,0)}) e_{(0,1)} +
  e_{(1,0)} e_{(0,1)},
\end{align}
However, the check at the order $u^2 v^2$ requires the evaluation of
the commutator $[e_{(2,0)}, e_{(0,2)}]$ which corresponds to a
non-admissible triangle. After sequential application of the
commutation relations for admissible triangles we find\footnote{There
  is an error in the first coefficient in the r.h.s.\ in this
  calculation on p.~30 of~\cite{BS}.}
\begin{equation}
  \label{eq:39}
  [e_{(2,0)}, e_{(0,2)}] = \frac{\kappa_1}{2} e_{(1,1)}^2 + \left(  2
    - \frac{\kappa_2}{2 \kappa_1} \right) e_{(2,2)},
\end{equation}
which indeed leads to a complete cancellation with other terms, so
that the pentagon identity holds at order $u^2 v^2$. The computations
for higher orders quickly become very involved. Probably they can be
computerized.

\section{Applications and implications}
\label{sec:appl-impl}
One can evaluate the identity~\eqref{eq:2} in other representation of
the DIM algebra, producing some interesting identities for operators
in representation spaces. Here we list a couple of preliminary
applications of the DIM pentagon identity along these lines.

\subsection{Pentagon identity for spherical double affine Hecke
  algebra from tensor product of vector representations}
\label{sec:pent-relat-spher}
There is a homomorphism from $U_{q_1, q_2,
  q_3}(\widehat{\widehat{\mathfrak{gl}}}_1)$ to spherical double
affine Hecke algebra (sDAHA) $\mathbb{H}_n$ with $n = 1, 2, \ldots$
The case $n=1$ is just the Reineke's character, or vector
representation $\mathcal{V}_{q_i}$ discussed in
sec.~\ref{sec:rein-homom-or}. For higher $n$ the homomorphism can be
described by noting that DIM algebra carries a coproduct structure (we
have not discussed it, since we did not use it for our checks of the
pentagon relation). This allows one to take a tensor product of $n$
copies of vector representation $\mathcal{V}_{q_i}$. This produces an
action of $U_{q_1, q_2, q_3}(\widehat{\widehat{\mathfrak{gl}}}_1)$ on
functions of $n$ coordinates $w_i$, involving Ruijsenaars-Schneider
difference operators. For example, in this representation we have
\begin{equation}
  \label{eq:41}
  \rho_i^{\otimes n} (E_k(e_{(m,0)})) \sim H_k^{\mathrm{RS}},
\end{equation}
where $H_k^{\mathrm{RS}}$ are Ruijsenaars-Schneider Hamiltonians
\begin{equation}
  \label{eq:44}
  H_k^{\mathrm{RS}} = \sum_{I \subset \{1,\ldots, n\}, |I|=k} \prod_{i \in I}
  \prod_{j \in \{1,\ldots,n\}\backslash I} \frac{t w_i - w_j}{w_i -
    w_j} q^{\sum_{i \in I} w_i \partial_{w_i}}
\end{equation}
and $E_k$ are complete symmetric functions.

Taking the limit $n \to \infty$ one recovers the Fock representation
and the identity of Garsia-Mellit with
\begin{equation}
  \label{eq:42}
  p_n = \sum_{i=1}^{\infty} w_i^n
\end{equation}
being power sum variables.

\subsection{Pentagon identity for Macmahon operators}
\label{sec:pent-ident-macm}
There is an interesting representation $m$ of $U_{q_1, q_2,
  q_3}(\widehat{\widehat{\mathfrak{gl}}}_1)$ on a vector space
$\mathcal{M}$ spanned by plane partitions (or $3d$ Young diagrams),
which is called the Macmahon representation. Most of the generators,
which we have described for Fock representation in
sec.~\ref{sec:pent-ident-gars} can also be written for Macmahon
representation. For example:
\begin{equation}
  \label{eq:43}
  m(e_{(n,0)})| \pi \rangle = \left(  - \frac{1}{\kappa_n} +
    \sum_{(i,j,k) \in \pi} q_1^{n(j-1)} q_2^{n(i-1)} q_3^{n(k-1)} \right)| \pi \rangle , \qquad n \geq 1
\end{equation}
where $\pi$ is a plane partition. The expressions for $m(e_{(0,n)})$
are more involved, but can also be found using certain analogues of
Pieri rules for plane partitions. It would be interesting to write
down and check the pentagon relation in this case.

\section{Conclusions and comments}
\label{sec:conclusions}

Let us recapitulate. We have conjectured an analogue of pentagon
relation between certain generating series in the DIM algebra. This
identity is very nontrivial, since the commutation relations of the
algebra are complicated with some of them defined implicitly by
recursion. We have performed a set of tests of the identity,
evaluating it in two families of representations and have found that
it reduces to previously known (and also nontrivial) identities in
these cases. We have also verified lower orders of the expansion in
the generating parameters using commutation relations directly,
without any reference to representations. Of course a formal proof of
Conjecture~\ref{conj:pent} would be desirable.

There is one puzzling fact about the identity~\eqref{eq:2}. The
generating series $T_{\vec{\gamma}}$ look like elements of a group and
therefore it would be logical to expect that they would have some
group-like property, i.e.\
\begin{equation}
  \label{eq:45}
  \Delta(T_{\vec{\gamma}}) \stackrel{?}{=} T_{\vec{\gamma}} \otimes T_{\vec{\gamma}}
\end{equation}
where $\Delta$ is the coproduct on a DIM algebra. However, a problem
arises immediately: exactly which coproduct features in
Eq.~\eqref{eq:45}. In fact the algebra $U_{q_1, q_2,
  q_3}(\widehat{\widehat{\mathfrak{gl}}}_1)$ has an infinite number of
coproduct structures labelled by an irrational slope in $\mathbb{Z}^2$
lattice. It turns out that Eq.~\eqref{eq:45} holds only for the
coproduct associated to the slope $\vec{\gamma} + 0$, where $+0$ means
the ``closest irrational slope'' to the rational vector
$\vec{\gamma}$.  For other choice of coproduct Eq.~\eqref{eq:45} is
manifestly wrong. How then would the identity involving
$T_{\vec{\gamma}}$ hold if some of the terms are group-like, while
others are not? For $q_3 \to 1$ this problem does not arise, since the
coproducts are all equivalent and trivial. One hope might be that
coproducts for different slopes $s$ and $s'$ are related by nontrivial
Drinfeld twists $F_{s, s'} \in U_{q_1, q_2,
  q_3}(\widehat{\widehat{\mathfrak{gl}}}_1) \otimes U_{q_1, q_2,
  q_3}(\widehat{\widehat{\mathfrak{gl}}}_1)$:
\begin{equation}
  \label{eq:46}
  \Delta_s = F_{s,s'} \Delta_{s'} F_{s,s'}^{-1}.
\end{equation}
The form of the Drinfeld twist $F_{s,s'}$ is known more or less
explicitly:
\begin{equation}
  \label{eq:47}
  F_{s,s'} \sim \prod_{s<\vec{\gamma}<s',\, \gcd(\vec{\gamma}) =1} \exp \left[ \sum_{n \geq 1}
    \kappa_n e_{n\vec{\gamma}} \otimes e_{-n\vec{\gamma}} \right],
\end{equation}
however, the commutation relations of $F_{(1,0),(0,1)}$ with
$T_{(1,0)}$ and $T_{(0,1)}$ seem to be complicated. Thus, the mystery
remains: how it happens that the coproduct preserves the pentagon
relation~\eqref{eq:2}?

Finally, it would be interesting to find if similar identities
involving $T_{\vec{\gamma}}$ with non-coprime $\vec{\gamma}$ hold. In
the limit $q_3 \to 1$ there are plenty of such identities each
corresponding to a motivic wall-crossing
formula~\cite{Kontsevich:2008fj}. One could wonder if any of them
survives $q_3$-deformation.

\section*{Acknowledgements}
This work is partly supported by the Russian Science Foundation (Grant
No.20-12-00195).

\end{document}